\documentclass[reqno,onecolumn,12pt]{amsart}
\usepackage{amsfonts,amscd,amsthm}
\usepackage{amssymb}
\usepackage{amsmath}
\usepackage{graphicx}
\usepackage{amscd,amsthm}

\newtheorem{theorem}{Theorem}
\theoremstyle{plain}

\newtheorem{corollary}{Corollary}

\newtheorem{remark}{Remark}

\numberwithin{equation}{section}

\input{tcilatex}

\begin{document}
\author{}
\title{}
\maketitle

\begin{center}
\pagestyle{myheadings}\thispagestyle{empty}%
\markboth{\bf Ilkay Arslan Guven and Semra Kaya Nurkan}
{\bf RULED SURFACES IN THREE DIMENSIONAL LIE GROUPS}

\textbf{\large RULED SURFACES IN THREE DIMENSIONAL LIE GROUPS}

\bigskip

\textbf{\.{I}lkay Arslan G\"{u}ven$^{1,\ast }$and Semra Kaya Nurkan$^{2}$}

\bigskip

$^{1}$Department of Mathematics, Faculty of Arts and Science, University of
Gaziantep, TR-27310 Gaziantep, Turkey

\textbf{E-Mail: iarslan@gantep.edu.tr, ilkayarslan81@hotmail.com}

\textbf{$^{\ast }$Corresponding Author}\\[2mm]

$^{2}$Department of Mathematics, Faculty of Arts and Science, University of U%
\c{s}ak,

TR-64200 U\c{s}ak, Turkey

\textbf{E-Mail: semra.kaya@usak.edu.tr, semrakaya\_gs@yahoo.com\\[2mm]
}

\textbf{\large Abstract}
\end{center}

\begin{quotation}
Motivated by a number of recent investigations, we define and investigate
the various properties of the ruled surfaces depend on three dimensional Lie
groups with a bi-variant metric. We give useful results involving the
characterizations of these ruled surfaces. Some special ruled surfaces such
as normal surface, binormal surface, tangent developable surface, rectifying
developable surface and Darboux developable surface are worked. From those
applications, we make use of such a work to interpret the Gaussian, mean
curvatures of these surfaces and geodesic, normal curvature and geodesic
torsion of the base curves with respect to these surfaces depend on three
dimensional Lie groups.
\end{quotation}

\noindent \textbf{2000 Mathematics Subject Classification.} 14J26, 22E15.

\noindent \textbf{Key Words and Phrases.} Ruled surface, Lie groups, Mean
curvatures, Normal curvature, Geodesic torsion.

\section{\textbf{Introduction}}

In the surface theory of geometry, ruled surfaces were found by French
mathematician Gaspard Monge who was a founder of constructive geometry.
Recently, many mathematicians have studied the ruled surfaces on Euclidean
space and Minkowski space for a long time. The information about these
topic, see, e.g., \cite{Ali, Izu, Izu2, Tur, Tur2, Yu} for a systematic work.

A \textit{ruled surface} in $\mathbb{R}^{3}$ is surface which can be
described as the set of points swept out by moving a straight line in
surface. It therefore has a parametrization of the form%
\begin{equation*}
\Phi (s,v)=\alpha (s)+v\delta (s)
\end{equation*}%
where $\alpha $ and $\delta $ are a curve lying on the surface called 
\textit{base curve} and \textit{director curve},\textit{\ }respectively. The
straight lines are called rulings. By using the equation of ruled surface we
assume that $\alpha ^{\prime }$ is never zero and $\delta $ is not
identically zero. The rulings of ruled surface are asymptotic curves.
Furthermore, the Gaussian curvature of ruled surface is everywhere
non-positive. The ruled surface is developable if and only if the
distribution parameter vanishes and it is minimal if and olny if its mean
curvature vanishes \cite{Gray}. A ruled surface is doubly ruled if through
every one of its points there are two distinct lines that lie on the
surface. Cylinder, cone, helicoid, Mobius strip, right conoid are some
examples of ruled surfaces and hyperbolic paraboloid and hyperboloid of one
sheet are doubly ruled surfaces.

Recently, there are many works about geometry and curve theory in three
dimensional Lie groups. \c{C}\"{o}ken and \c{C}ift\c{c}i studied the
degenerate semi- Riemannian geometry of Lie Gruops. They found reductive
homogeneous semi-Riemannian space from the Lie group in a natural way \cite%
{Cok}. Next, general helices in three dimensional Lie group with
bi-invariant metric are defined by \c{C}ift\c{c}i in \cite{Cif}. He
generalized the Lancret's theorem and obtained so-called spherical general
helices, and also he gave a relation between the geodesics of the so-called
cylinders and general helices.

In \cite{Cif}, a cylinder which is a surface was defined in a three
dimensional Lie group with a bi-variant metric in accordance with the
definition of a ruled surface in Riemannian manifold. If $G$ is a three
dimensional Lie group and $\mathfrak{g}$ is its Lie algebra, then a cylinder
is a surface $\varphi (t,\lambda )$ given by $\varphi :\mathbb{R}\times 
\mathbb{R}\longrightarrow G$, \ $\varphi (t,\lambda )=\alpha (t)\exp
(\lambda X)$, \ where $\alpha :\mathbb{R}\longrightarrow G$ is a curve in $G$%
, $X\in \mathfrak{g}$ \ and \ 
\begin{equation*}
\exp :\mathfrak{g}\longrightarrow G
\end{equation*}%
\ is the exponential mapping of $G$.

Meeks and P\'{e}rez studied geometry of constant mean curvature $H\geq 0$
surfaces which are called H-surfaces in three dimensional simply-connected
Lie group see\cite{Meek}.

Slant helices in three dimensional Lie groups were defined by Okuyucu et al.
in \cite{Ok}. They obtained a characterization of slant helices and gave
some relations between slant helices and their involutes, spherical images.
They also defined Bertrand curves and Mannheim curves in three dimensional
lie groups in \cite{Ok3, Ok2} and gave the harmonic curvature function for
some special curves such as helix, slant curves, Mannheim curves and
Bertrand curves.

In the present paper, we define and investigate the ruled surface in three
dimensional Lie groups with a bi-variant metric. We obtain the Gaussian and
mean curvatures, distribution parameter of the ruled surface. Also we find
the geodesic, normal curvatures and geodesic torsion of the base curve of
ruled surface with respect to ruled surface in three dimensional Lie groups.
In the final part of this paper, we give some characterizations of the ruled
surface using the curvatures.

\section{Preliminaries}

A \textit{Lie group }is a nonempty subset $G$ which satisfies the following
conditions;

\textbf{1) }$G$ $\ $is a group.

\textbf{2) }$G$ is a smooth manifold.

\textbf{3) }$G$ \ is a topological group, in particular, the group operation
\ $\circ :G$ $\times G$ $\longrightarrow G$ \ and \ the inverse map $inv:G$ $%
\longrightarrow G$ \ are smooth.

Let $\mathfrak{g}$ \ be the \textit{Lie algebra }of $G.$ $\mathfrak{g}$ \ is
a vector space together with a bilinear map%
\begin{equation*}
\left[ \text{ },\right] :\mathfrak{g\times g\longrightarrow g}
\end{equation*}%
called Lie bracket on $\mathfrak{g,}$ such that the following two identities
hold for all $a,b,c\in \mathfrak{g}$%
\begin{equation*}
\left[ a,a\right] =0
\end{equation*}%
and the so-called Jacobi identity 
\begin{equation*}
\left[ a,\left[ b,c\right] \right] +\left[ c,\left[ a,b\right] \right] +%
\left[ b,\left[ c,a\right] \right] =0.
\end{equation*}%
It is immediately verified that $\left[ a,b\right] =-\left[ b,a\right] .$

If $G$ $\ $is a Lie group, a vector field $X$ on $G$ $\ $is \textit{%
left-invariant}, if 
\begin{equation*}
d(L_{a})_{b}(X(b))=X(L_{a})_{b})=X(ab)
\end{equation*}%
for all $a,b\in G$. Here $L_{a}:G\longrightarrow G$ \ and \ $%
d(L_{a}):T_{G}\longrightarrow T_{G}$ \ where $T_{G}$ \ is a tangent vector
space.

Similarly $X$ is \textit{right-invariant}, if%
\begin{equation*}
d(R_{a})_{b}(X(b))=X(R_{a})_{b})=X(ba).
\end{equation*}%
A Riemannian metric on a Lie group $G$ \ is called left-invariant if%
\begin{equation*}
\left\langle u,v\right\rangle =\left\langle
d(L_{a})(u),d(L_{a})(v)\right\rangle
\end{equation*}%
where $u,v\in T_{G}(a),$ \ $a\in G$ . A metric on $G$ that is both
left-invariant and right-invariant is called \textit{bi-invariant} \ (see 
\cite{Kar}).

Let G be a Lie group with bi-invariant metric $\left\langle ,\right\rangle $
and let $D$ be the corresponding Levi-Civita connection. If $\mathfrak{g}$
is the Lie algebra of $G$, then $\mathfrak{g}$ \ is isomorph to $T_{e}G$
where $e$ is the neutral element of $G$. For the bi-variant metric $%
\left\langle ,\right\rangle $, we have%
\begin{equation}
\left\langle X,\left[ Y,Z\right] \right\rangle =\left\langle \left[ X,Y%
\right] ,Z\right\rangle  \tag{2.1}
\end{equation}%
\begin{equation}
D_{X}Y=\frac{1}{2}\left[ X,Y\right]  \tag{2.2}
\end{equation}%
for all $X,Y,Z\in \mathfrak{g}$ .

Let $\alpha :I\subset \mathbb{R}\longrightarrow G$ \ be a parametrized curve
and $\left\{ X_{1},X_{2},...,X_{n}\right\} $ be an orthonormal basis of $%
\mathfrak{g}$. We can write two vector fields \ $W$ and $Z$ as $W=\underset{%
i=1}{\overset{n}{\sum }}\omega _{i}X_{i}$ \ and \ $Z=\underset{i=1}{\overset{%
n}{\sum }}z_{i}X_{i}$ where $\omega _{i}:I\longrightarrow \mathbb{R}$ \ and
\ $z_{i}:I\longrightarrow \mathbb{R}$ are smooth functions. The Lie bracket
of $W$ and $Z$ is defined by $\left[ W,Z\right] =\underset{i,j=1}{\overset{n}%
{\sum }}\omega _{i}z_{j}\left[ X_{i},X_{j}\right] $. If the directional
derivative of $W$ is $\overset{\cdot }{W}=\underset{i=1}{\overset{n}{\sum }}%
\overset{\cdot }{\omega _{i}}X_{i}$ \ for $\overset{\cdot }{\omega _{i}}=%
\frac{d\omega }{dt}$, \ then the following equation hold as;%
\begin{equation}
D_{\alpha ^{\prime }}W=\overset{\cdot }{W}+\frac{1}{2}\left[ T,W\right] 
\tag{2.3}
\end{equation}%
where $\alpha ^{\prime }=T$ is the tangent vector field of $\alpha $. Note
that if $W$ is left-invariant vector field of $\alpha $ , then $\overset{%
\cdot }{W}=0$ \ (see \cite{Cro, Cif})

Now, let $\alpha $ be a parametrized curve in three dimensional Lie group $G$
\ and $\left\{ T,N,B,\kappa ,\tau \right\} $ be the Frenet apparatus of the
curve $\alpha $. Then \c{C}ift\c{c}i \cite{Cif} defined $\tau _{G}$ as;%
\begin{equation}
\tau _{G}=\frac{1}{2}\left\langle \left[ T,N\right] ,B\right\rangle 
\tag{2.4}
\end{equation}%
or%
\begin{equation*}
\tau _{G}=\frac{1}{2\kappa ^{2}\tau }\left\langle \overset{\cdot \cdot }{T},%
\left[ T,\overset{\cdot }{T}\right] \right\rangle +\frac{1}{4\kappa ^{2}\tau 
}\left\Vert \left[ T,\overset{\cdot }{T}\right] \right\Vert ^{2}.
\end{equation*}%
Also the following equalities were given in \cite{Ok};%
\begin{eqnarray}
\left[ T,N\right] &=&2\tau _{G}B  \TCItag{2.5} \\
\left[ T,B\right] &=&-2\tau _{G}N  \notag
\end{eqnarray}%
By using the equations (2.3) and (2.5), the Frenet formulas for the curve $%
\alpha $ are given as 
\begin{eqnarray}
D_{T}T &=&\kappa N  \notag \\
D_{T}N &=&-\kappa T+(\tau +\tau _{G})B  \TCItag{2.6} \\
D_{T}B &=&-(\tau +\tau _{G})N  \notag
\end{eqnarray}%
After some computation which we use equations (2.3) and (2.6), the curvature 
$\kappa $ and torsion $\tau $ are found by%
\begin{eqnarray}
\kappa &=&\left\Vert D_{T}T\right\Vert =\left\Vert \overset{\cdot }{T}%
\right\Vert  \notag \\
\tau &=&\left\Vert D_{T}B\right\Vert -\tau _{G}  \TCItag{2.7}
\end{eqnarray}%
(for curvature $\kappa $ see \cite{Cif}).

It is known that cross product $\times $ in $\mathbb{R}^{3}$ is a Lie \
bracket. If the three dimensional special orthogonal group with the
bi-variant metric is $SO(3)$, then by identifying $\mathfrak{so(3)}$ \ with $%
(\mathbb{R}^{3},\times )$, we have $\left[ X,Y\right] =X\times Y$ for all $%
X,Y\in \mathfrak{so(3)}$. So for a curve in $SO(3)$, it is shown that (see 
\cite{Cif}) 
\begin{equation}
\tau _{G}=\frac{1}{2}\left\langle T\times N,B\right\rangle =\frac{1}{2}. 
\tag{2.8}
\end{equation}%
Also if $G$ is Abelian ,then $\tau _{G}=0$ (see \cite{Cif}).

\section{\textbf{Ruled Surfaces In Three Dimensional Lie Groups}}

We will define ruled surfaces in three dimensional Lie groups .Then we will
obtain the distribution parameter, Gaussian curvature and mean curvature of
these ruled surfaces. Also we will identify the geodesic curvature, the
normal curvature and geodesic torsion of the base curve of ruled surfaces.

\textbf{Definition 3.1}: Let $G$ be the three dimensional Lie group with a
bi-invariant metric $\left\langle ,\right\rangle .$ A \textit{ruled surface }
$\varphi (s,v)$\ in $G$ $\ $, $\varphi :\mathbb{R\times R\longrightarrow }G$%
, is given by%
\begin{equation}
\varphi (s,v)=\alpha (s)+vX(s)  \tag{3.1}
\end{equation}%
where $\alpha :\mathbb{R\longrightarrow }G$ \ is called \textit{base curve }%
and $X\in \mathfrak{g}$ \ is a left-invariant unit vector field which is
called \textit{director}. The directors denote straight lines which are
called \textit{rulings} of the ruled surface.

The base curve $\alpha $ is given with the arc-length parameter $s$, the set 
$\left\{ T,N,B,\kappa ,\tau \right\} $ denote the Frenet apparatus of $%
\alpha $ , $\alpha ^{\prime }=T$ , $\kappa \neq 0$ \ and $\tau _{G}=\frac{1}{%
2}\left\langle \left[ T,N\right] ,B\right\rangle .$

\textbf{Definition 3.2}: If there exists a common perpendicular to two
constructive rulings in the surface, then the foot of the common
perpendicular on the main ruling is called \textit{central point}. The locus
of the central point is called \textit{striction curve}. \ The striction
curve of the ruled surface $\varphi $ in three dimensional Lie group $G$ is
given by%
\begin{equation}
\overline{\alpha }=\alpha -\frac{\left\langle \alpha ^{\prime
},D_{T}X\right\rangle }{\left\Vert D_{T}X\right\Vert ^{2}}X.  \tag{3.2}
\end{equation}

\textbf{Definition 3.3}: The \textit{distribution parameter} $\lambda $ of
the ruled surface $\varphi $ in three dimensional Lie group $G$ given by
equation (3.1) is dedicated as;%
\begin{equation}
\lambda =\frac{\det (T,X,D_{T}X)}{\left\Vert D_{T}X\right\Vert ^{2}}. 
\tag{3.3}
\end{equation}

The standard unit normal vector field $U$ on the ruled surface $\varphi $ \
is defined by%
\begin{equation}
U=\frac{\varphi _{s}\times \varphi _{v}}{\left\Vert \varphi _{s}\times
\varphi _{v}\right\Vert }.  \tag{3.4}
\end{equation}
where $\varphi _{s}=\frac{d\varphi }{ds}$ \ and \ $\varphi _{v}=\frac{%
d\varphi }{dv}$.

\textbf{Definition 3.4}: The \textit{Gaussian curvature} and \textit{mean
curvature} of the ruled surface $\varphi $ in three dimensional Lie group $G$%
\ are given respectively by%
\begin{equation}
K=\frac{eg-f^{2}}{EG-F^{2}}  \tag{3.5}
\end{equation}%
and%
\begin{equation}
H=\frac{Eg+Ge-2Ff}{2(EG-F^{2})}  \tag{3.6}
\end{equation}%
where $E=\left\langle \varphi _{s},\varphi _{s}\right\rangle $ \ , \ $%
F=\left\langle \varphi _{s},\varphi _{v}\right\rangle $ \ , \ $%
G=\left\langle \varphi _{v},\varphi _{v}\right\rangle $ \ , \ $%
e=\left\langle \varphi _{ss},U\right\rangle $ \ , \ $f=\left\langle \varphi
_{sv},U\right\rangle $ \ and $g=\left\langle \varphi _{vv},U\right\rangle .$

\textbf{Definition 3.5}: \textbf{\ }For a surface $\Phi $ in three
dimensional Lie group $G,$

\textbf{1) }$\Phi $ is \textit{developable} if and only if the distribution
parameter of $\Phi $ vanishes.

\textbf{2) }$\Phi $ is called \textit{minimal} if and if only the mean
curvature of $\Phi $ vanishes.

\textbf{Definition 3.6}: If the Gaussian curvature of a surface in in three
dimensional Lie group $G$ is $K$, then

\textbf{1) }If $K\langle 0$ , then a point on the surface is hyperbolic.

\textbf{2) }If $K=0$ , then a point on the surface is parabolic.

\textbf{3) }If $K\rangle 0$ , then a point on the surface is elliptic.

\textbf{Definition 3.7}: If the curve $\alpha $ is the base curve of the
ruled surface $\varphi $ in three dimensional Lie group $G$, then the
geodesic curvature, normal curvature and geodesic torsion with respesct to
the ruled surface $\varphi $ are computed as follows;%
\begin{equation}
\kappa _{g_{\varphi }}=\left\langle U\times T,D_{T}T\right\rangle  \tag{3.7}
\end{equation}%
\begin{equation}
\kappa _{n_{\varphi }}=\left\langle D_{T}T,U\right\rangle  \tag{3.8}
\end{equation}%
and%
\begin{equation}
\tau _{g_{\varphi }}=\left\langle U\times D_{T}U,D_{T}T\right\rangle . 
\tag{3.9}
\end{equation}

(For the formulas of $\kappa _{g}$, $\kappa _{n}$ and $\tau _{g}$ in
Euclidean space see \cite{Ali}).

\begin{remark}
Note that the curvatures and torsion of the curve $\alpha $ in equations
(3.7), (3.8) and (3.9) are computed with respect to ruled surface $\varphi $
and the geodesic torsion $\tau _{G}$ in equation (2.4) of \ $\alpha $ is
given with respect to three dimensional Lie group $G$.
\end{remark}

\textbf{Definition 3.8}: For a curve $\beta $ which is lying on a surface in
three dimensional Lie group $G,$ the following statements are satisfied;

\textbf{1) }$\beta $ is a geodesic curve if and only if the geodesic
curvature of the curve with respect to the surface vanishes.

\textbf{2) }$\beta $ is a asymptotic line if and only if the normal
curvature of the curve with respect to the surface vanishes.

\textbf{3) }$\beta $ is a principal line if and only if the geodesic torsion
of the curve with respect to the surface vanishes.

\begin{theorem}
Let $\varphi (s,v)=\alpha (s)+vX(s)$ be a ruled surface in three dimensional
Lie group $G$ \ with unit left-invariant vector field $X$, $\alpha :\mathbb{%
R\longrightarrow }G$ \ be the base curve and $\left\{ T,N,B,\kappa ,\tau
\right\} $ be the Frenet apparatus of $\alpha $ . The base curve $\alpha $
is always the striction curve of the ruled surface $\varphi $.
\end{theorem}

\begin{proof}
If we use the equation (3.2) and make the appropriate calculations , we find
the striction curve as;%
\begin{eqnarray*}
\overline{\alpha } &=&\alpha -\frac{\left\langle \alpha ^{\prime
},D_{T}X\right\rangle }{\left\Vert D_{T}X\right\Vert ^{2}}X \\
&=&\alpha -\frac{\left\langle T,\frac{1}{2}\left[ T,X\right] \right\rangle }{%
\left\Vert \left[ T,X\right] \right\Vert ^{2}}X \\
&=&\alpha -\frac{1}{2}\frac{\left\langle \left[ T,T\right] ,X\right\rangle }{%
\left\Vert \left[ T,X\right] \right\Vert ^{2}}X \\
&=&\alpha .
\end{eqnarray*}
\end{proof}

\begin{theorem}
Let $\varphi (s,v)=\alpha (s)+vX(s)$ be a ruled surface in three dimensional
Lie group $G$ \ with unit left-invariant vector field $X$, $\alpha :\mathbb{%
R\longrightarrow }G$ \ be the base curve and $\left\{ T,N,B,\kappa ,\tau
\right\} $ be the Frenet apparatus of $\alpha $ .The distribution parameter,
the Gaussian curvature and the mean curvature of $\varphi $ are given
respesctively as;%
\begin{equation*}
\lambda =2\frac{\left\langle T\times X,\left[ T,X\right] \right\rangle }{%
\left\Vert \left[ T,X\right] \right\Vert ^{2}}
\end{equation*}%
\begin{equation*}
K=-\frac{\left\langle T\times X,\left[ T,X\right] \right\rangle ^{2}}{%
4A^{2}(1+\frac{v^{2}}{4}\left\Vert \left[ T,X\right] \right\Vert
^{2}-\left\langle T,X\right\rangle ^{2})}
\end{equation*}%
and%
\begin{equation*}
H=\frac{%
\begin{array}{c}
\frac{1}{A}\mathbf{(}-\kappa \left\langle B,X\right\rangle -\frac{v\kappa }{2%
}\left\langle N\times X,\left[ T,X\right] \right\rangle +\frac{v\kappa }{2}%
\left\langle \left[ N,X\right] ,T\times X,\right\rangle \\ 
+\frac{v^{2}\kappa }{4}\left\langle \left[ N,X\right] ,\left[ T,X\right]
\times X,\right\rangle +\frac{1}{2}\left\langle \left[ T,\left[ T,X\right] %
\right] ,T\times X\right\rangle \\ 
+\frac{v}{4}\left\langle \left[ T,\left[ T,X\right] \right] ,\left[ T,X%
\right] \times X\right\rangle )-\frac{1}{A}\left\langle T,X\right\rangle
\left\langle T\times X,\left[ T,X\right] \right\rangle%
\end{array}%
}{2(1+\frac{v^{2}}{4}\left\Vert \left[ T,X\right] \right\Vert
^{2}-\left\langle T,X\right\rangle ^{2})}
\end{equation*}%
where $A=\left\Vert \varphi _{s}\times \varphi _{v}\right\Vert $.
\end{theorem}

\begin{proof}
If $\varphi (s,v)=\alpha (s)+vX(s)$ is a ruled surface in three dimensional
Lie group $G,$ then we can compute%
\begin{eqnarray*}
E &=&1+\frac{v^{2}}{4}\left\Vert \left[ T,X\right] \right\Vert ^{2}\text{ \
\ \ \ , \ \ \ \ \ }F=\left\langle T,X\right\rangle \text{ \ \ \ \ \ \ , \ \
\ \ \ \ }G=1 \\
e &=&\frac{1}{A}\left( 
\begin{array}{c}
-\kappa \left\langle B,X\right\rangle -\frac{v\kappa }{2}\left\langle
N\times X,\left[ T,X\right] \right\rangle +\frac{v\kappa }{2}\left\langle %
\left[ N,X\right] ,T\times X,\right\rangle \\ 
+\frac{v^{2}\kappa }{4}\left\langle \left[ N,X\right] ,\left[ T,X\right]
\times X,\right\rangle +\frac{1}{2}\left\langle \left[ T,\left[ T,X\right] %
\right] ,T\times X\right\rangle \\ 
+\frac{v}{4}\left\langle \left[ T,\left[ T,X\right] \right] ,\left[ T,X%
\right] \times X\right\rangle%
\end{array}%
\right) \\
f &=&\frac{1}{2A}\left\langle T\times X,\left[ T,X\right] \right\rangle 
\text{ \ \ \ \ , \ \ \ \ \ }g=0.
\end{eqnarray*}%
where $A=\left\Vert \varphi _{s}\times \varphi _{v}\right\Vert .$ By using
the equations (3.5) and (3.6), we easily find Gaussian and mean curvatures .

Also with the equations (2.3) and (3.3), distribution parameter is obtained
directly.
\end{proof}

\begin{corollary}
The ruled surface $\varphi (s,v)=\alpha (s)+vX(s)$ in three dimensional Lie
group $G$ is developable if and only if the vector fields $T\times X$ and $%
\left[ T,X\right] $ are orthogonal. The ruled surface $\varphi $ is minimal
if and only if the following equation is satisfied;%
\begin{equation*}
\begin{array}{c}
-\kappa \left\langle B,X\right\rangle -\frac{v\kappa }{2}\left\langle
N\times X,\left[ T,X\right] \right\rangle +\frac{v\kappa }{2}\left\langle %
\left[ N,X\right] ,T\times X,\right\rangle \\ 
+\frac{v^{2}\kappa }{4}\left\langle \left[ N,X\right] ,\left[ T,X\right]
\times X,\right\rangle +\frac{1}{2}\left\langle \left[ T,\left[ T,X\right] %
\right] ,T\times X\right\rangle \\ 
+\frac{v}{4}\left\langle \left[ T,\left[ T,X\right] \right] ,\left[ T,X%
\right] \times X\right\rangle \\ 
=\left\langle T,X\right\rangle \left\langle T\times X,\left[ T,X\right]
\right\rangle%
\end{array}%
\end{equation*}
\end{corollary}

\begin{proof}
By using the definition (3.5) and the distribution parameter, the mean
curvature which are found in above theorem the results are apparent.
\end{proof}

\begin{remark}
Notice that if $\Phi $ is a ruled surface in Euclidean space, then $K\leq 0$
where $K$ is the Gaussian curvature of $\Phi $. Altough $K\leq 0$ for the
ruled surface $\Phi $ in Euclidean space, it is not always true for a ruled
surface in three dimensional Lie group .
\end{remark}

\begin{theorem}
Let $\varphi (s,v)=\alpha (s)+vX(s)$ be a ruled surface in three dimensional
Lie group $G$ \ with unit left-invariant vector field $X$, $\alpha :\mathbb{%
R\longrightarrow }G$ \ be the base curve and $\left\{ T,N,B,\kappa ,\tau
\right\} $ be the Frenet apparatus of $\alpha .$ The geodesic curvature,
normal curvature and geodesic torsion of $\alpha $ with respect to ruled
surface $\varphi $ are given respectively as;%
\begin{equation*}
\kappa _{g_{\varphi }}=\frac{\kappa }{A}(\left\langle X,N\right\rangle
+v\tau _{G}\left\langle X,B\right\rangle )
\end{equation*}%
\begin{equation*}
\kappa _{n_{\varphi }}=\frac{\kappa }{A}(-\left\langle X,B\right\rangle +%
\frac{v}{2}\left\langle \left[ T,X\right] ,X\times N\right\rangle )
\end{equation*}%
and%
\begin{eqnarray*}
\tau _{g_{\varphi }} &=&\left\langle X,N\right\rangle (\frac{\kappa }{A^{2}}%
(\kappa \left\langle X,B\right\rangle +\frac{v}{2}\left\langle T,L\times
X\right\rangle ) \\
&&+\frac{v\kappa }{2A^{2}}(\kappa \left\langle \left[ T,X\right] ,N\times
X\right\rangle +\frac{v}{2}\left\langle \left[ T,X\right] ,L\times
X\right\rangle ) \\
&&+\frac{1}{2A^{2}}(\frac{v\kappa }{2}\left\langle \left[ \left[ T,X\right]
,T\right] ,T\times X\right\rangle +\frac{v^{2}\kappa }{4}\left\langle \left[ %
\left[ T,X\right] ,T\right] ,\left[ T,X\right] \times X\right\rangle )) \\
&&-\frac{v\kappa \tau _{G}}{A^{2}}\left\langle \left[ T,X\right] ,T\times
X\right\rangle \left\langle X,B\right\rangle
\end{eqnarray*}%
where $A=\left\Vert \varphi _{s}\times \varphi _{v}\right\Vert $ , $\tau
_{G}=\frac{1}{2}\left\langle \left[ T,N\right] ,B\right\rangle $ $\ $and $\
L=\kappa \left[ X,N\right] +\frac{1}{2}\left[ T,\left[ T,X\right] \right] .$
\end{theorem}

\begin{proof}
If the equation of ruled surface is $\varphi (s,v)=\alpha (s)+vX(s)$, then
the unit normal vector field of $\varphi $ is found as; 
\begin{equation*}
U=\frac{1}{A}(T\times X+\frac{v}{2}\left[ T,X\right] \times X).
\end{equation*}%
By using the equation (2.3), we have%
\begin{eqnarray*}
D_{T}U &=&\overset{\cdot }{U}+\frac{1}{2}\left[ T,U\right] \\
&=&\left( \frac{1}{A}\right) ^{\prime }(T\times X+\frac{v}{2}\left[ T,X%
\right] \times X)+\frac{1}{A}(\kappa (N\times X)+\frac{1}{2}(T\times \left[
T,X\right] ) \\
&&+\frac{v}{2}((\kappa \left[ N,X\right] +\frac{1}{2}\left[ T,\left[ T,X%
\right] \right] )\times X)) \\
&&+\frac{1}{2A}(\left[ T,T\times X\right] +\frac{v}{2}\left[ T,\left[ T,X%
\right] \times X\right] ).
\end{eqnarray*}%
If we use the equations (2.6), (3.7), (3.8) , (3.9) and make the appropriate
calculations, the proof is completed.
\end{proof}

\begin{corollary}
If the director vector field $X$ and the binormal vector field $B$ are
ortogonal and $\left\langle \left[ T,X\right] ,X\times N\right\rangle =0$
then the base curve $\alpha $ of $\varphi $ is a asymptotic line.
\end{corollary}

\begin{proof}
If we consider the definition (3.8) and \ the normal curvature $\kappa
_{n_{\varphi }}$ given in theorem the result is clear.
\end{proof}

\begin{corollary}
If the director vector field $X$ $\ $is orthogonal to both the principal
normal vector field $N$ and the binormal vector field $B$, then the base
curve $\alpha $ of $\varphi $ is geodesic curve and principal line.
\end{corollary}

\begin{proof}
If the director vector field $X$ $\ $is orthogonal to both the principal
normal vector field $N$ and the binormal vector field $B$, then 
\begin{equation*}
\left\langle X,N\right\rangle =0\text{ \ \ \ \ and \ \ \ \ }\left\langle
X,B\right\rangle =0.
\end{equation*}%
By using geodesic curvature and geodesic torsion given in the theorem, we get%
\begin{equation*}
\kappa _{g_{\varphi }}=0\text{ \ \ \ \ and \ \ \ \ }\tau _{g_{\varphi }}=0.
\end{equation*}%
These equations denotes that $\alpha $ of $\varphi $ is geodesic curve and
principal line, by the definition (3.8).
\end{proof}

\textbf{Example : }Let a ruled surface which is a cylinder in three
dimensional Lie group $G$, is given with the equation%
\begin{equation*}
\varphi (t,v)=(\cos t,\sin t,0)+v(0,0,1).
\end{equation*}%
The Frenet vector fields of the base curve $\alpha (t)=$ $(\cos t,\sin t,0)$
are $T=(-\sin t,\cos t,0)$ , \ $N=(-\cos t,-\sin t,0)$ \ and \ $B=(0,0,1).$

Since the curve $\alpha (t)=$ $(\cos t,\sin t,0)$ is also a circle in $%
\mathbb{R}^{3}$, we can compute $\tau _{G}=\frac{1}{2}\left\langle \left[ T,N%
\right] ,B\right\rangle =\frac{1}{2}\left\langle T\times N,B\right\rangle =%
\frac{1}{2}.$ By the equations in (2.7), curvature and torsion of $\alpha $
are found as $\kappa =\left\Vert \overset{\cdot }{T}\right\Vert =1$ \ and%
\begin{eqnarray*}
\tau &=&\left\Vert \overset{\cdot }{B}+\frac{1}{2}\left[ T,B\right]
\right\Vert -\tau _{G} \\
&=&\left\Vert (0,0,0)+(\frac{1}{2}\cos t,\frac{1}{2}\sin t,0)\right\Vert -%
\frac{1}{2} \\
&=&0.
\end{eqnarray*}

For the curvatures we find the following expressions%
\begin{eqnarray*}
\left\langle X,N\right\rangle &=&0\text{ \ , \ }\left\langle
X,B\right\rangle =1\text{ \ , \ }\left\langle T,X\right\rangle =0\text{ \ \
\ , \ \ \ \ }A=1 \\
\left[ T,X\right] &=&T\times X=(\cos t,\sin t,0)\text{ \ \ , \ \ }\left[ T,%
\left[ T,X\right] \right] =(0,0,1) \\
\left[ N,X\right] &=&(-\sin t,\cos t,0)\text{ \ \ , \ \ }\left[ T,X\right]
\times X=(\sin t,-\cos t,0).
\end{eqnarray*}%
Now by using the expressions above, the distribution parameter, Gaussian
curvature and mean curvature are obtained as follows;%
\begin{eqnarray*}
\lambda &=&2 \\
K &=&-\frac{1}{v^{2}+4} \\
H &=&-\frac{v^{2}+2}{v^{2}+4}.
\end{eqnarray*}%
Also the geodesic curvature, normal curvature and geodesic torsion of $%
\alpha $ with respect to the cylinder are%
\begin{eqnarray*}
\kappa _{g_{\varphi }} &=&\frac{v}{2} \\
\kappa _{n_{\varphi }} &=&-1 \\
\tau _{g_{\varphi }} &=&-\frac{v}{2}.
\end{eqnarray*}

\begin{remark}
Notice that a cylinder in in Euclidean space is developable but a cylinder
in three dimensional Lie group $G$ is not developable.
\end{remark}

\section{\textbf{Some Special} \textbf{Ruled Sufaces In Three Dimensional
Lie Groups}}

In this section, we will identify some special ruled surfaces which are
existed \ in Euclidean space. For details of these surfaces see \cite{Gray,
Izu, Izu2}.

\textbf{Definition 4.1: }Let $G$ be the three dimensional Lie group with
bi-invariant metric and $\alpha :\mathbb{R\longrightarrow }G$ be a
parametrized curve with the Frenet apparatus $\left\{ T,N,B,\kappa ,\tau
\right\} $ , the modified Darboux vector field $W=\frac{1}{\sqrt{\kappa
^{2}+\tau ^{2}}}(\tau T+\kappa B)$ \ , $\alpha ^{\prime }=T$ $\ $, $\ \kappa
\neq 0$ \ and $\tau _{G}=\frac{1}{2}\left\langle \left[ T,N\right]
,B\right\rangle .$ Some types of ruled surfaces in three dimensional Lie
group $G$ are defined and given with their equations as follows;

\textbf{1)} Tangent developable surface ; \ $\varphi (s,v)=\alpha
(s)+vT(s)\qquad \qquad $(4.1)

\textbf{2)} Normal surface; $\varphi (s,v)=\alpha (s)+vN(s)\qquad \qquad
\qquad \qquad \qquad \ $(4.2)

\textbf{3)} Binormal surface; $\varphi (s,v)=\alpha (s)+vB(s)\qquad \qquad
\qquad \qquad \ \ \ \ \ $(4.3)

\textbf{4)} Darboux developable surface; $\varphi (s,v)=B(s)+vT(s)\qquad \ \
\ \ \ \ \ $(4.4)

\textbf{5)} Rectifying surface; $\varphi (s,v)=\alpha (s)+vW(s).\qquad
\qquad \qquad \qquad \ \ \ $(4.5)

\begin{theorem}
Let $\varphi (s,v)=\alpha (s)+vT(s)$ be a tangent developable surface in
three dimensional Lie group $G$ . The distribution parameter, Gaussian
curvature and mean curvature of the surface $\varphi $ are given by%
\begin{eqnarray*}
\lambda &=&0 \\
K &=&0 \\
H &=&-\frac{\tau +\tau _{G}}{2v^{2}\kappa }
\end{eqnarray*}%
and the geodesic curvature, normal curvature, geodesic torsion of $\alpha $
with respect to tangent developable surface are%
\begin{eqnarray*}
\kappa _{g_{\varphi }} &=&-\kappa \\
\kappa _{n_{\varphi }} &=&0 \\
\tau _{g_{\varphi }} &=&0.
\end{eqnarray*}
\end{theorem}

\begin{proof}
For the tangent developable surface given in equation (4.1), the following
expressions are computed as; \ $E=1+v^{2}\kappa ^{2}$ \ , \ $F=1$ \ , \ $G=1$
\ , \ $e=-\kappa (\tau +\tau _{G})$ \ , \ $f=0$ $\ $, $\ g=0$ and the normal
vector field of the surface by the equation (3.4) is $U=-B.$ By using the
equations (3.3), (3.5), (3.6), (3.7), (3.8) and (3.9) the results are
obtained clearly.
\end{proof}

\begin{corollary}
The tangent developable surface in three dimensional Lie group $G$ is
developable and not minimal. A point on this surface is parabolic. The base
curve $\alpha $ on the surface is asymptotic and principal line but it is
not geodesic curve.
\end{corollary}

\begin{proof}
Since the distribution parameter of tangent developable surface is zero,
then it is developable. If we pay attention to the Frenet formulas in
equation (2.6), the mean curvature can not be zero because of $\ \tau
_{G}\neq -\tau .$ Also by definition (3.6), a point on the surface is
parabolic.

By thinking the definition (3.8) and since $\kappa \neq 0$ , the base curve $%
\alpha $ is asymptotic and principal line and not geodesic curve.
\end{proof}

\begin{theorem}
Let $\varphi (s,v)=\alpha (s)+vN(s)$ be a normal surface in three
dimensional Lie group $G$ . The distribution parameter, Gaussian curvature
and mean curvature of the surface $\varphi $ are given by%
\begin{eqnarray*}
\lambda &=&\frac{\tau +\tau _{G}}{\kappa ^{2}+(\tau +\tau _{G})^{2}} \\
K &=&-\left( \frac{\tau +\tau _{G}}{A^{2}}\right) ^{2} \\
H &=&-\frac{v(\tau +\tau _{G})(1-v\kappa +v\kappa ^{\prime })}{2A^{3}}
\end{eqnarray*}%
and the geodesic curvature, normal curvature, geodesic torsion of $\alpha $
with respect to normal surface are%
\begin{eqnarray*}
\kappa _{g_{\varphi }} &=&\frac{\kappa (1-v\kappa )}{A} \\
\kappa _{n_{\varphi }} &=&0 \\
\tau _{g_{\varphi }} &=&\kappa \left( \frac{v(\tau +\tau _{G})}{A}.\left( 
\frac{1-v\kappa }{A}\right) ^{\prime }-\frac{1-v\kappa }{A}.\left( \frac{%
v(\tau +\tau _{G})}{A}\right) ^{\prime }\right)
\end{eqnarray*}%
where $A=\sqrt{v^{2}(\tau +\tau _{G})^{2}+(1-v\kappa )^{2}}.$
\end{theorem}

\begin{proof}
For the normal surface given in equation (4.2), the following expressions
are computed as; \ $E=$ $v^{2}(\tau +\tau _{G})^{2}+(1-v\kappa )^{2}$\ , \ $%
F=0$ \ , \ $G=1$ \ , \ $e=\frac{v(\tau +\tau _{G})(1-v\kappa +v\kappa
^{\prime })}{A}$ \ , \ $f=\frac{\tau +\tau _{G}}{A}$ $\ $, $\ g=0$ \ and the
normal vector field of the surface by the equation (3.4) is found as; 
\begin{equation*}
U=\frac{1}{A}(-v(\tau +\tau _{G})T+(1-v\kappa )B).
\end{equation*}
By using the equations (3.3), (3.5), (3.6), (3.7), (3.8) and (3.9) the
results are obtained clearly.
\end{proof}

\begin{corollary}
The normal surface in three dimensional Lie group $G$ is not developable. It
is minimal if and only if \ the equation $v\kappa -v\kappa ^{\prime }=1$ is
satisfied. A point on this surface is hyperbolic. The base curve $\alpha $
on the surface is asymptotic line. $\alpha $ is geodesic curve if and only
if $\ v\kappa =1$ \ and it is principal line if and only if $\ $%
\begin{equation*}
\frac{v(\tau +\tau _{G})}{A}.\left( \frac{1-v\kappa }{A}\right) ^{\prime }=%
\frac{1-v\kappa }{A}.\left( \frac{v(\tau +\tau _{G})}{A}\right) ^{\prime }.
\end{equation*}
\end{corollary}

\begin{proof}
Since $\tau _{G}\neq -\tau $ , then the distribution parameter cannot be
zero, so the normal surface is not developable. By the mean curvature found
in theorem and the definition (3.5) and since $v\neq 0$, $\tau _{G}\neq
-\tau $, the surface is minimal with the satisfied equation $1-v\kappa
+v\kappa ^{\prime }=0.$ Also by definition (3.6), a point on the surface is
hyperbolic.

By deciding the definition (3.8) and $\kappa \neq 0,$ \ since $\kappa
_{n_{\varphi }}=0$ , then $\alpha $ \ is asymptotic line. $v\kappa =1$ if
and only if $\kappa _{g_{\varphi }}=0$ ,so $\alpha $ is$\ $ geodesic curve. $%
\tau _{g_{\varphi }}=0$ if and only if%
\begin{equation*}
\frac{v(\tau +\tau _{G})}{A}.\left( \frac{1-v\kappa }{A}\right) ^{\prime }=%
\frac{1-v\kappa }{A}.\left( \frac{v(\tau +\tau _{G})}{A}\right) ^{\prime }.
\end{equation*}%
So $\alpha $ is principal line with the satisfied equation above.
\end{proof}

\begin{theorem}
Let $\varphi (s,v)=\alpha (s)+vB(s)$ be a binormal surface in three
dimensional Lie group $G$ . The distribution parameter, Gaussian curvature
and mean curvature of the surface $\varphi $ are given by%
\begin{eqnarray*}
\lambda &=&\frac{1}{\tau +\tau _{G}} \\
K &=&-\left( \frac{\tau +\tau _{G}}{A^{2}}\right) ^{2} \\
H &=&-\frac{-v^{2}\kappa (\tau +\tau _{G})+v\tau ^{\prime }-\kappa }{2A^{3}}
\end{eqnarray*}%
and the geodesic curvature, normal curvature, geodesic torsion of $\alpha $
with respect to binormal surface are%
\begin{eqnarray*}
\kappa _{g_{\varphi }} &=&\frac{\kappa }{A} \\
\kappa _{n_{\varphi }} &=&-\frac{\kappa }{A} \\
\tau _{g_{\varphi }} &=&\frac{v\kappa (\tau +\tau _{G})(A(\tau +\tau
_{G})-\tau _{G})}{A^{2}}
\end{eqnarray*}%
where $A=\sqrt{1+v^{2}(\tau +\tau _{G})^{2}}.$
\end{theorem}

\begin{proof}
For the binormal surface given in equation (4.3), the following expressions
are computed as; \ $E=1+$ $v^{2}(\tau +\tau _{G})^{2}$\ , \ $F=0$ \ , \ $G=1$
\ , \ $e=\frac{-v^{2}\kappa (\tau +\tau _{G})+v\tau ^{\prime }-\kappa }{A}$
\ , \ $f=\frac{\tau +\tau _{G}}{A}$ $\ $, $\ g=0$ \ and the normal vector
field of the surface by the equation (3.4) is found as; 
\begin{equation*}
U=-\frac{1}{A}(v(\tau +\tau _{G})T+N).
\end{equation*}%
By using the equations (3.3), (3.5), (3.6), (3.7), (3.8) and (3.9) the
results are obtained clearly.
\end{proof}

\begin{corollary}
The binormal surface in three dimensional Lie group $G$ is not developable.
It is minimal if and only if \ the equation $v^{2}\kappa (\tau +\tau
_{G})=v\tau ^{\prime }-\kappa $ is satisfied. A point on this surface is
hyperbolic. The base curve $\alpha $ on the surface is not geodesic curve
and asymptotic line. $\alpha $ is principal line if and only if 
\begin{equation*}
\frac{\tau _{G}^{2}-v^{2}(\tau +\tau _{G})^{4}}{(\tau +\tau _{G})^{2}}=1.
\end{equation*}
\end{corollary}

\begin{proof}
Since $\lambda \neq 0$ , the binormal surface is not developable. By using
the definition (3.5), the surface is minimal with the satisfied equation $%
v^{2}\kappa (\tau +\tau _{G})=v\tau ^{\prime }-\kappa .$ Also by definition
(3.6), a point on the surface is hyperbolic.

Since $\kappa \neq 0$, then $\kappa _{g_{\varphi }}\neq 0$ \ and \ $\kappa
_{n_{\varphi }}\neq 0$. Also since $v\neq 0$, $\kappa \neq 0$ $\ $and $\
\tau _{G}\neq -\tau $, then $\tau _{g_{\varphi }}=0$ if and only if $A(\tau
+\tau _{G})=\tau _{G}.$ If we make necessary calculations in the equation $%
A(\tau +\tau _{G})=\tau _{G}$, then we get $\frac{\tau _{G}^{2}-v^{2}(\tau
+\tau _{G})^{4}}{(\tau +\tau _{G})^{2}}=1.$
\end{proof}

\begin{theorem}
Let $\varphi (s,v)=B(s)+vT(s)$ be a Darboux developable surface in three
dimensional Lie group $G$ . The distribution parameter, Gaussian curvature
and mean curvature of the surface $\varphi $ are given by%
\begin{eqnarray*}
\lambda &=&0 \\
K &=&0 \\
H &=&\frac{1}{2(\tau +\tau _{G}-v\kappa )}
\end{eqnarray*}%
and the geodesic curvature, normal curvature, geodesic torsion of $\alpha $
with respect to Darboux developable surface are%
\begin{eqnarray*}
\kappa _{g_{\varphi }} &=&\kappa \\
\kappa _{n_{\varphi }} &=&0 \\
\tau _{g_{\varphi }} &=&0
\end{eqnarray*}
\end{theorem}

\begin{proof}
For the Darboux developable surface given in equation (4.4), the following
expressions are computed as; \ $E=$ $(v\kappa -(\tau +\tau _{G}))^{2}$\ , \ $%
F=0$ \ , \ $G=1$ \ , \ $e=v\kappa -(\tau +\tau _{G})$ \ , \ $f=0$ $\ $, $\
g=0$ \ and the normal vector field of the surface by the equation (3.4) is
found as; 
\begin{equation*}
U=B.
\end{equation*}%
By using the equations (3.3), (3.5), (3.6), (3.7), (3.8) and (3.9) the
results are obtained clearly.
\end{proof}

\begin{corollary}
The Darboux developable surface in three dimensional Lie group $G$ is
developable. It is not minimal . A point on this surface is parabolic. The
base curve $\alpha $ on the surface is asymptotic line and principal line
but it is not geodesic curve.
\end{corollary}

\begin{proof}
By the definition (3.5) and since the distribution parameter $\lambda =0,$
the Darboux developable surface is developable. Since the mean curvature
cannot be zero, ,it is not minimal. Also by definition (3.6), a point on the
surface is parabolic.

If we use \ the definition (3.8) and since $\kappa \neq 0$, $\ \kappa
_{g_{\varphi }}\neq 0$, then the base curve $\alpha $ is not geodesic curve.
Also $\alpha $ is asymptotic line and principal line because of $\kappa
_{n_{\varphi }}=0$ and $\tau _{g_{\varphi }}=0$.
\end{proof}

\begin{theorem}
Let $\varphi (s,v)=\alpha (s)+vW(s)$ be a rectifying surface in three
dimensional Lie group $G$ . The distribution parameter, Gaussian curvature
and mean curvature of the surface $\varphi $ are given by%
\begin{eqnarray*}
\lambda &=&\frac{c^{2}\kappa ^{2}\tau _{G}}{((c\tau )^{\prime
})^{2}+((c\kappa )^{\prime })^{2}+c^{2}\kappa ^{2}\tau _{G}^{2}} \\
K &=&-\frac{1}{A^{2}}\cdot \frac{\left( c^{2}\kappa ^{2}\tau
_{G}(1+v((c\kappa )^{\prime }-(c\tau )^{\prime })\right) )^{2}}{(1+v(c\kappa
)^{\prime })^{2}+v^{2}(c\kappa )^{\prime 2}+(vc\kappa \tau
_{G})^{2}-(1+vc(c\kappa )^{\prime }(\kappa +\tau )} \\
&& \\
&& \\
H &=&\frac{1}{A}\cdot \frac{%
\begin{array}{c}
-v^{2}c^{2}\kappa \tau _{G}(c\kappa )^{\prime \prime }(\kappa +\tau
)-v^{2}c(\kappa -\tau )^{2}(c\kappa )^{\prime 2}-2v^{2}c\tau _{G}(c\kappa
)^{\prime 2}(\kappa +\tau ) \\ 
+2vc\kappa (c\kappa )^{\prime }(\tau -\kappa -\tau _{G})-c\kappa
^{2}+v^{2}c^{3}\kappa ^{2}\tau _{G}^{2}(\tau ^{2}-\kappa ^{2}+\tau \tau _{G})
\\ 
-2(c\tau +vc(c\kappa )^{\prime }(\kappa +\tau ))(vc^{2}\kappa ^{2}\tau
_{G}((c\kappa )^{\prime }-(c\tau )^{\prime })+c^{2}\kappa ^{2}\tau _{G})%
\end{array}%
}{2((1+v(c\kappa )^{\prime })^{2}+(v(c\kappa )^{\prime
})^{2}+v^{2}c^{2}\kappa ^{2}\tau _{G}^{2}-(c\tau +vc(c\kappa )^{\prime
}(\kappa +\tau ))^{2})}
\end{eqnarray*}%
and the geodesic curvature, normal curvature, geodesic torsion of $\alpha $
with respect to rectifying surface are%
\begin{eqnarray*}
\kappa _{g_{\varphi }} &=&\frac{vc^{2}\kappa ^{2}\tau \tau _{G}}{A} \\
\kappa _{n_{\varphi }} &=&\frac{\kappa (vc(c\kappa )^{\prime }(\tau -\kappa
)-c\kappa )}{A} \\
\tau _{g_{\varphi }} &=&\frac{\kappa }{A}.\left( 
\begin{array}{c}
vc^{2}\kappa \tau \tau _{G}\left( \left( \frac{-vc^{2}\kappa ^{2}\tau _{G}}{A%
}\right) ^{\prime }-\frac{vc\kappa (c\kappa )^{\prime }(\tau -\kappa
)-c\kappa ^{2}}{A}\right) \\ 
-vc^{2}\kappa ^{2}\tau _{G}\left( \frac{(vc(c\kappa )^{\prime }(\tau -\kappa
)-c\kappa )(\tau +2\tau _{G})}{A}+\left( \frac{vc^{2}\kappa \tau \tau _{G}}{A%
}\right) ^{\prime }\right)%
\end{array}%
\right)
\end{eqnarray*}%
where $c=\frac{1}{\sqrt{\kappa ^{2}+\tau ^{2}}}$ $\ $and $\ A=\sqrt{%
v^{2}c^{4}\kappa ^{2}\tau _{G}^{2}(\kappa ^{2}+\tau ^{2})+(vc(c\kappa
)^{\prime }(\tau -\kappa )-c\kappa )^{2}}.$
\end{theorem}

\begin{proof}
For the rectifying surface given in equation (4.5), the following
expressions are computed as;%
\begin{eqnarray*}
E &=&(1+v(c\kappa )^{\prime })^{2}+(v(c\kappa )^{\prime })^{2}+(vc\kappa
\tau _{G})^{2} \\
F &=&c\tau +vc(c\kappa )^{\prime }(\kappa +\tau ) \\
G &=&1
\end{eqnarray*}%
\ and%
\begin{eqnarray*}
e &=&%
\begin{array}{c}
-v^{2}c^{2}\kappa \tau _{G}(c\kappa )^{\prime \prime }(\kappa +\tau
)-v^{2}c(\kappa -\tau )^{2}(c\kappa )^{\prime 2}-2v^{2}c\tau _{G}(c\kappa
)^{\prime 2}(\kappa +\tau ) \\ 
+2vc\kappa (c\kappa )^{\prime }(\tau -\kappa -\tau _{G})-c\kappa
^{2}+v^{2}c^{3}\kappa ^{2}\tau _{G}^{2}(\tau ^{2}-\kappa ^{2}+\tau \tau _{G})%
\end{array}
\\
&& \\
f &=&c^{2}\kappa ^{2}\tau _{G}(1+v((c\kappa )^{\prime }-(c\tau )^{\prime }))
\\
g &=&0.
\end{eqnarray*}%
\ The normal vector field of the surface by the equation (3.4) is found as; 
\begin{equation*}
U=\frac{1}{A}\left( 
\begin{array}{c}
-vc^{2}\kappa ^{2}\tau _{G}T+(vc(c\kappa )^{\prime }(\tau -\kappa )-c\kappa
)N \\ 
vc^{2}\kappa \tau \tau _{G}B%
\end{array}%
\right) .
\end{equation*}%
By using the equations (3.3), (3.5), (3.6), (3.7), (3.8), (3.9) and making
necessary calculations and simplifications ,the results are obtained clearly.
\end{proof}

\begin{corollary}
If $G$ is abelian, then the rectifying surface in three dimensional Lie
group $G$ is developable and a point on this surface is parabolic. If $G$ is
abelian or the base curve $\alpha $ on the surface is plane curve , then $%
\alpha $ is geodesic curve. $\alpha $ is asymptotic line and principal line
if and only if the following equations satisfy respectively;%
\begin{equation*}
vc(c\kappa )^{\prime }(\tau -\kappa )=c\kappa
\end{equation*}%
and%
\begin{equation*}
\begin{array}{c}
vc^{2}\kappa \tau \tau _{G}\left( \left( \frac{-vc^{2}\kappa ^{2}\tau _{G}}{A%
}\right) ^{\prime }-\frac{vc\kappa (c\kappa )^{\prime }(\tau -\kappa
)-c\kappa ^{2}}{A}\right) \\ 
-vc^{2}\kappa ^{2}\tau _{G}\left( \frac{(vc(c\kappa )^{\prime }(\tau -\kappa
)-c\kappa )(\tau +2\tau _{G})}{A}+\left( \frac{vc^{2}\kappa \tau \tau _{G}}{A%
}\right) ^{\prime }\right)%
\end{array}%
=0.
\end{equation*}
\end{corollary}

\begin{proof}
If $G$ is abelian, then $\tau _{G}=0($see\cite{Cif}). Since $\tau _{G}=0$, $%
c\neq 0$ \ and \ $\kappa \neq 0$, the distribution parameter is zero. So the
surface is developable. Also if $\tau _{G}=0$, then the Gaussian curvature
is zero, this means that a point on this surface is parabolic.

If $G$ is abelian \ or $\tau =0,$ and $v\neq 0,c\neq 0$ \ and \ $\kappa \neq
0$ ,then geodesic curvature is zero. The normal curvature and geodesic
torsion are zero if and only if the equations in corollary satisfy.
\end{proof}

\begin{remark}
Altough a rectifying surface with the equation $\varphi (s,v)=\alpha
(s)+vW(s)$ in three dimensional Lie group $G$ is not developable, it is
developable in Euclidean space.
\end{remark}

\end{document}